\documentclass{amsart}
\usepackage{amsmath}
\usepackage{amsfonts}

\newtheorem{theorem}{Theorem}
\theoremstyle{plain}

\newtheorem{corollary}{Corollary}

\newtheorem{definition}{Definition}

\newtheorem{lemma}{Lemma}

\newtheorem{remark}{Remark}

\numberwithin{equation}{section}
\input{tcilatex}

\begin{document}
\title[On generalization of different type inequalities]{Generalization of
different type integral inequalities for $s-$convex functions via fractional
integrals}
\author{\.{I}mdat \.{I}\c{s}can}
\address{Department of Mathematics, Faculty of Sciences and Arts, Giresun
University, Giresun, Turkey}
\email{imdat.iscan@giresun.edu.tr}

\begin{abstract}
In this paper, a general integral identity for twice differentiable
functions is derived. By using of this identity, the author establish some
new estimates on Hermite-Hadamard type and Simpson type inequalities for $s-$%
convex via Riemann Liouville fractional integral.
\end{abstract}

\maketitle

\section{Introduction}

Following inequalities are well known in the literature as Hermite-Hadamard
inequality and Simpson inequality respectively:

\begin{theorem}
Let $f:I\subseteq \mathbb{R\rightarrow R}$ be a convex function defined on
the interval $I$ of real numbers and $a,b\in I$ with $a<b$. The following
double inequality holds%
\begin{equation}
f\left( \frac{a+b}{2}\right) \leq \frac{1}{b-a}\dint\limits_{a}^{b}f(x)dx%
\leq \frac{f(a)+f(b)}{2}\text{.}  \label{1-1}
\end{equation}
\end{theorem}

\begin{theorem}
Let $f:\left[ a,b\right] \mathbb{\rightarrow R}$ be a four times
continuously differentiable mapping on $\left( a,b\right) $ and $\left\Vert
f^{(4)}\right\Vert _{\infty }=\underset{x\in \left( a,b\right) }{\sup }%
\left\vert f^{(4)}(x)\right\vert <\infty .$ Then the following inequality
holds:%
\begin{equation*}
\left\vert \frac{1}{3}\left[ \frac{f(a)+f(b)}{2}+2f\left( \frac{a+b}{2}%
\right) \right] -\frac{1}{b-a}\dint\limits_{a}^{b}f(x)dx\right\vert \leq 
\frac{1}{2880}\left\Vert f^{(4)}\right\Vert _{\infty }\left( b-a\right) ^{2}.
\end{equation*}
\end{theorem}

In \cite{HM94}, Hudzik and Maligranda considered among others the class of
functions which are $s-$convex in the second sense.

\begin{definition}
A function $f:\left[ 0,\infty \right) \rightarrow 
\mathbb{R}
$ is said to be $s-$convex in the second sense if%
\begin{equation*}
f(\alpha x+\beta y)\leq \alpha ^{s}f(x)+\beta ^{s}f(y)
\end{equation*}%
for all $x,y\in \lbrack 0,\infty )$, $\alpha ,\beta \geq 0$ with $\alpha
+\beta =1$ and for some fixed $s\in (0,1]$. This class of $s-$convex
functions in the second sense is usually denoted by $K_{s}^{2}$.
\end{definition}

It can be easily seen that for $s=1$, $s-$convexity reduces to ordinary
convexity of functions defined on $[0,\infty )$.

We give some necessary definitions and mathematical preliminaries of
fractional calculus theory which are used throughout this paper.

\begin{definition}
Let $f\in L\left[ a,b\right] $. The Riemann-Liouville integrals $%
J_{a^{+}}^{\alpha }f$ and $J_{b^{-}}^{\alpha }f$ of oder $\alpha >0$ with $%
a\geq 0$ are defined by

\begin{equation*}
J_{a^{+}}^{\alpha }f(x)=\frac{1}{\Gamma (\alpha )}\dint\limits_{a}^{x}\left(
x-t\right) ^{\alpha -1}f(t)dt,\ x>a
\end{equation*}

and

\begin{equation*}
J_{b^{-}}^{\alpha }f(x)=\frac{1}{\Gamma (\alpha )}\dint\limits_{x}^{b}\left(
t-x\right) ^{\alpha -1}f(t)dt,\ x<b
\end{equation*}%
respectively, where $\Gamma (\alpha )$ is the Gamma function defined by $%
\Gamma (\alpha )=$ $\dint\limits_{0}^{\infty }e^{-t}t^{\alpha -1}dt$ and $%
J_{a^{+}}^{0}f(x)=J_{b^{-}}^{0}f(x)=f(x).$
\end{definition}

In the case of $\alpha =1$, the fractional integral reduces to the classical
integral. Properties concerning this operator can be found \cite%
{GM97,MR93,P99}.

In \cite{SS10}, Sarikaya et.al established the following theorem:

\begin{theorem}
Let $f:$ $I\subset 
\mathbb{R}
\rightarrow 
\mathbb{R}
$ be a twice differentiable mapping on $I^{\circ }$ such that $f^{\prime
\prime }\in L[a,b]$, where $a,b\in I$ with $a<b$. If $|f^{\prime \prime
}|^{q}$ is convex mapping on $[a,b]$ for $q\geq 1$, then the following
inequality holds:%
\begin{eqnarray}
&&\left\vert \frac{1}{6}\left[ f(a)+4f\left( \frac{a+b}{2}\right) +f(b)%
\right] -\frac{1}{\left( b-a\right) }\dint\limits_{a}^{b}f(x)dx\right\vert 
\notag \\
&\leq &\frac{\left( b-a\right) ^{2}}{162}\left\{ \left( \frac{59\left\vert
f^{\prime \prime }\left( a\right) \right\vert ^{q}+133\left\vert f^{\prime
\prime }\left( b\right) \right\vert ^{q}}{192}\right) ^{\frac{1}{q}}\right. 
\notag \\
&&\left. +\left( \frac{59\left\vert f^{\prime \prime }\left( b\right)
\right\vert ^{q}+133\left\vert f^{\prime \prime }\left( a\right) \right\vert
^{q}}{192}\right) ^{\frac{1}{q}}\right\} .  \label{1-2}
\end{eqnarray}
\end{theorem}

In \cite{XQ13}, Bo-Yan Xi and Feng Qi establish some new integral
inequalities of Hermite-Hadamard type for h-convex functions as follow:

\begin{theorem}
\label{0.1}Let $I,J\subset 
\mathbb{R}
$ be intervals, $\left( 0,1\right) \subset J,$ and $h:J\rightarrow \left[
0,\infty \right) .$ Let $f:$ $I\subset 
\mathbb{R}
\rightarrow 
\mathbb{R}
$ be a twice differentiable mapping on $I^{\circ }$ such that $f^{\prime
\prime }\in L[a,b]$ for $a,b\in I$ with $a<b$. If $\ 0\leq \lambda \leq 1$
and $|f^{\prime \prime }|^{q}$ is an $h-$convex mapping on $[a,b]$ for $%
q\geq 1$, then 
\begin{eqnarray*}
&&\left\vert \left( 1-\lambda \right) f\left( \frac{a+b}{2}\right) +\lambda
\left( \frac{f(a)+f(b)}{2}\right) -\frac{1}{\left( b-a\right) }%
\dint\limits_{a}^{b}f(x)dx\right\vert \\
&\leq &\frac{\left( b-a\right) ^{2}}{16}\left[ H_{1}(\lambda )\right] ^{1-%
\frac{1}{q}}\left\{ \left( \left\vert f^{\prime \prime }\left( \frac{a+b}{2}%
\right) \right\vert ^{q}\dint\limits_{0}^{1}t\left\vert 2\lambda
-t\right\vert h(t)dt+\left\vert f^{\prime \prime }\left( a\right)
\right\vert ^{q}\dint\limits_{0}^{1}t\left\vert 2\lambda -t\right\vert
h(1-t)dt\right) ^{\frac{1}{q}}\right. \\
&&+\left. \left( \left\vert f^{\prime \prime }\left( \frac{a+b}{2}\right)
\right\vert ^{q}\dint\limits_{0}^{1}t\left\vert 2\lambda -t\right\vert
h(t)dt+\left\vert f^{\prime \prime }\left( b\right) \right\vert
^{q}C\dint\limits_{0}^{1}t\left\vert 2\lambda -t\right\vert h(1-t)dt\right)
^{\frac{1}{q}}\right\} ,
\end{eqnarray*}%
where%
\begin{equation*}
H_{1}(\lambda )=\left\{ 
\begin{array}{cc}
\frac{8\lambda ^{3}-3\lambda +1}{3}, & 0\leq \lambda \leq \frac{1}{2} \\ 
\frac{3\lambda -1}{3}, & \frac{1}{2}<\lambda \leq 1%
\end{array}%
\right.
\end{equation*}
\end{theorem}

Let us consider the following special functions:

(1) The Beta function:%
\begin{equation*}
\beta \left( x,y\right) =\frac{\Gamma (x)\Gamma (y)}{\Gamma (x+y)}%
=\dint\limits_{0}^{1}t^{x-1}\left( 1-t\right) ^{y-1}dt,\ \ x,y>0,
\end{equation*}%
(2) The incomplete Beta function: 
\begin{equation*}
\beta \left( a,x,y\right) =\dint\limits_{0}^{a}t^{x-1}\left( 1-t\right)
^{y-1}dt,\ \ 0<a<1,\ x,y>0,
\end{equation*}%
(3) The hypergeometric function:%
\begin{equation*}
_{2}F_{1}\left( a,b;c;z\right) =\frac{1}{\beta \left( b,c-b\right) }%
\dint\limits_{0}^{1}t^{b-1}\left( 1-t\right) ^{c-b-1}\left( 1-zt\right)
^{-a}dt,\ c>b>0,\ \left\vert z\right\vert <1\text{ (see \cite{AS65}).}
\end{equation*}%
In Theorem \ref{0.1}, if we take $h(t)=t^{s},s\in \left( 0,1\right] ,$ we
have the following inequality%
\begin{eqnarray}
&&\left\vert \left( 1-\lambda \right) f\left( \frac{a+b}{2}\right) +\lambda
\left( \frac{f(a)+f(b)}{2}\right) -\frac{1}{\left( b-a\right) }%
\dint\limits_{a}^{b}f(x)dx\right\vert   \notag \\
&\leq &\frac{\left( b-a\right) ^{2}}{16}\left[ H_{1}(\lambda )\right] ^{1-%
\frac{1}{q}}\left\{ \left( \left\vert f^{\prime \prime }\left( \frac{a+b}{2}%
\right) \right\vert ^{q}H_{2}(\lambda ,s)+\left\vert f^{\prime \prime
}\left( a\right) \right\vert ^{q}H_{3}(\lambda ,s)\right) ^{\frac{1}{q}%
}\right.   \notag \\
&&+\left. \left( \left\vert f^{\prime \prime }\left( \frac{a+b}{2}\right)
\right\vert ^{q}H_{2}(\lambda ,s)+\left\vert f^{\prime \prime }\left(
b\right) \right\vert ^{q}H_{3}(\lambda ,s)\right) ^{\frac{1}{q}}\right\} ,
\label{1-3}
\end{eqnarray}%
where%
\begin{equation*}
H_{2}(\lambda ,s)=\left\{ 
\begin{array}{cc}
\frac{2\left( 2\lambda \right) ^{s+3}}{\left( s+2\right) \left( s+3\right) }-%
\frac{2\lambda }{s+2}+\frac{1}{s+3}, & 0\leq \lambda \leq \frac{1}{2} \\ 
\frac{2\lambda }{s+2}-\frac{1}{s+3}, & \frac{1}{2}<\lambda \leq 1%
\end{array}%
\right. ,
\end{equation*}%
and%
\begin{equation*}
H_{3}(\lambda ,s)=\left\{ 
\begin{array}{cc}
\left[ 
\begin{array}{c}
2\lambda \beta \left( 2,s+1\right) -\beta \left( 3,s+1\right)  \\ 
+4\lambda \beta \left( 2\lambda ;2,s+1\right) -2\beta \left( 2\lambda
;3,s+1\right) 
\end{array}%
\right] , & 0\leq \lambda \leq \frac{1}{2} \\ 
2\lambda \beta \left( 2,s+1\right) -\beta \left( 3,s+1\right) , & \frac{1}{2}%
<\lambda \leq 1%
\end{array}%
\right. ,
\end{equation*}%
In \cite{P12}, Park establihed some new inequalities of the Simpson's like
and the Hermite-Hadamar-like type for $s-$convex functions as follows:

\begin{theorem}
Let $f:$ $I\subset \lbrack 0,\infty )\rightarrow 
\mathbb{R}
$ be a twice differentiable function on $I^{\circ }$ such that $f^{\prime
\prime }$ is integrable on $[a,b]$, where $a,b\in I^{\circ }$ with $a<b$. If 
$|f^{\prime \prime }|^{q}$ is $s-$convex on $[a,b]$ for some fixed $s\in
\left( 0,1\right] $ and $q>1,$ then for $r\geq 2$ the following inequality
holds:

(a) 
\begin{eqnarray}
&&\frac{16}{\left( b-a\right) ^{2}}\left\vert H_{a}^{b}(f)(r)\right\vert 
\label{1-4} \\
&\leq &\left\{ \left[ \frac{r-2}{r}\right] ^{2p+1}\beta \left(
1+p,1+p\right) +\frac{2^{p+1}}{r^{p+1}\left( p+1\right) }._{2}F_{1}\left(
-p,1;p+2;\frac{2}{r}\right) \right\} ^{\frac{1}{p}}  \notag \\
&&\times \left\{ \left( \frac{\left\vert f^{\prime \prime }\left( \frac{a+b}{%
2}\right) \right\vert ^{q}+\left\vert f^{\prime \prime }\left( a\right)
\right\vert ^{q}}{s+1}\right) ^{\frac{1}{q}}+\left( \frac{\left\vert
f^{\prime \prime }\left( \frac{a+b}{2}\right) \right\vert ^{q}+\left\vert
f^{\prime \prime }\left( b\right) \right\vert ^{q}}{s+1}\right) ^{\frac{1}{q}%
}\right\} ,  \notag
\end{eqnarray}%
where%
\begin{eqnarray*}
&&H_{a}^{b}(f)(r) \\
&=&\left( \frac{1}{2}-\frac{1}{r}\right) \left\{ \frac{f(a)+f(b)}{2}\right\}
+\left( \frac{1}{2}+\frac{1}{r}\right) f\left( \frac{a+b}{2}\right) -\frac{1%
}{b-a}\dint\limits_{a}^{b}f(x)dx.
\end{eqnarray*}%
(b)%
\begin{eqnarray*}
&&\frac{16}{\left( b-a\right) ^{2}}\left\vert R_{a}^{b}(f)(r)\right\vert  \\
&\leq &\left( \frac{2}{r}\right) .\left( \frac{_{2}F_{1}\left( -p,1;p+2;%
\frac{-r}{2}\right) }{\left( p+1\right) }\right) ^{\frac{1}{p}} \\
&&\times \left\{ \left( \frac{\left\vert f^{\prime \prime }\left( \frac{a+b}{%
2}\right) \right\vert ^{q}+\left\vert f^{\prime \prime }\left( a\right)
\right\vert ^{q}}{s+1}\right) ^{\frac{1}{q}}+\left( \frac{\left\vert
f^{\prime \prime }\left( \frac{a+b}{2}\right) \right\vert ^{q}+\left\vert
f^{\prime \prime }\left( b\right) \right\vert ^{q}}{s+1}\right) ^{\frac{1}{q}%
}\right\} ,
\end{eqnarray*}%
where%
\begin{eqnarray*}
&&R_{a}^{b}(f)(r) \\
&=&\left( \frac{1}{2}+\frac{1}{r}\right) \left\{ \frac{f(a)+f(b)}{2}\right\}
+\left( \frac{1}{2}-\frac{1}{r}\right) f\left( \frac{a+b}{2}\right) -\frac{1%
}{b-a}\dint\limits_{a}^{b}f(x)dx.
\end{eqnarray*}
\end{theorem}

In recent years, many athors have studied errors estimations for
Hermite-Hadamard's inequality and Simpson's on the class of $s-$convex
mappings in the second sense and $(\alpha ,m)-$convex mappings; for
refinements, counterparts, generalizations and new Simpson's type
inequalities, see \cite{DBI13,OAK11,P12,P12b,SA11,SS10,XQ13}.

The main aim of this article is to establish new generalization of Hermite
Hadamard-type and Simpson-type inequalities for functions whose absolute
values of second derivatives are $s-$convex. To begin with, the author will
derive a general integral identity for twice differentiable mappings. By
using this integral equality, the author establish some new inequalities of
the Simpson-like and the Hermite-Hadamard-like type for these functions.\ 

\section{Main Results}

Let $f:I\subseteq 
\mathbb{R}
\rightarrow 
\mathbb{R}
$ be a differentiable function on $I^{\circ }$, the interior of $I$,
throughout this section we will take%
\begin{eqnarray*}
&&I_{f}\left( x,\lambda ,\alpha ,a,b\right) \\
&=&\left( 1-\lambda \right) \left[ \frac{\left( x-a\right) ^{\alpha }+\left(
b-x\right) ^{\alpha }}{b-a}\right] f(x)+\lambda \left[ \frac{\left(
x-a\right) ^{\alpha }f(a)+\left( b-x\right) ^{\alpha }f(b)}{b-a}\right] \\
&&+\left( \frac{1}{\alpha +1}-\lambda \right) \left[ \frac{\left( b-x\right)
^{\alpha +1}-\left( x-a\right) ^{\alpha +1}}{b-a}\right] f^{\prime }(x)-%
\frac{\Gamma \left( \alpha +1\right) }{b-a}\left[ J_{x^{-}}^{\alpha
}f(a)+J_{x^{+}}^{\alpha }f(b)\right]
\end{eqnarray*}%
where $a,b\in I$ with $a<b$, $\ x\in \lbrack a,b]$ , $\lambda \in \left[ 0,1%
\right] $, $\alpha >0$ and $\Gamma $ is Euler Gamma function. In order to
prove our main results we need the following identity.

\begin{lemma}
\label{2.1}Let $f:I\subseteq 
\mathbb{R}
\rightarrow 
\mathbb{R}
$ be a twice differentiable function on $I^{\circ }$ such that $f^{\prime
\prime }\in L[a,b]$, where $a,b\in I$ with $a<b$. Then for all $x\in \lbrack
a,b]$ , $\lambda \in \left[ 0,1\right] $ and $\alpha >0$ we have:%
\begin{eqnarray}
&&I_{f}\left( x,\lambda ,\alpha ,a,b\right) =\frac{\left( x-a\right)
^{\alpha +2}}{\left( \alpha +1\right) \left( b-a\right) }\dint%
\limits_{0}^{1}t\left( \left( \alpha +1\right) \lambda -t^{\alpha }\right)
f^{\prime \prime }\left( tx+\left( 1-t\right) a\right) dt  \label{2-1} \\
&&+\frac{\left( b-x\right) ^{\alpha +2}}{\left( \alpha +1\right) \left(
b-a\right) }\dint\limits_{0}^{1}t\left( \left( \alpha +1\right) \lambda
-t^{\alpha }\right) f^{\prime \prime }\left( tx+\left( 1-t\right) b\right)
dt.  \notag
\end{eqnarray}
\end{lemma}

\begin{proof}
By integration by parts twice and changing the variable, for $x\neq a,$ we
can state 
\begin{eqnarray}
&&\dint\limits_{0}^{1}t\left( \left( \alpha +1\right) \lambda -t^{\alpha
}\right) f^{\prime \prime }\left( tx+\left( 1-t\right) a\right) dt
\label{2-1a} \\
&=&\left. t\left( \left( \alpha +1\right) \lambda -t^{\alpha }\right) \frac{%
f^{\prime }\left( tx+\left( 1-t\right) a\right) }{x-a}\right\vert _{0}^{1}-%
\frac{\alpha +1}{x-a}\dint\limits_{0}^{1}\left( \lambda -t^{\alpha }\right)
f^{\prime }\left( tx+\left( 1-t\right) a\right) dt  \notag \\
&=&\left( \left( \alpha +1\right) \lambda -1\right) \frac{f^{\prime }\left(
x\right) }{x-a}-\frac{\alpha +1}{x-a}\left[ \left. \left( \lambda -t^{\alpha
}\right) \frac{f\left( tx+\left( 1-t\right) a\right) }{x-a}\right\vert
_{0}^{1}+\dint\limits_{0}^{1}\alpha t^{\alpha -1}\frac{f\left( tx+\left(
1-t\right) a\right) }{x-a}dt\right]  \notag \\
&=&\left( \left( \alpha +1\right) \lambda -1\right) \frac{f^{\prime }\left(
x\right) }{x-a}+\frac{\alpha +1}{\left( x-a\right) ^{2}}\left[ \left(
1-\lambda \right) f(x)+\lambda f(a)-\frac{\alpha }{\left( x-a\right)
^{\alpha }}\dint\limits_{a}^{x}\left( u-a\right) ^{\alpha -1}f(u)du\right] 
\notag \\
&=&\left( \left( \alpha +1\right) \lambda -1\right) \frac{f^{\prime }\left(
x\right) }{x-a}+\frac{\alpha +1}{\left( x-a\right) ^{2}}\left[ \left(
1-\lambda \right) f(x)+\lambda f(a)-\frac{\Gamma \left( \alpha +1\right) }{%
\left( x-a\right) ^{\alpha }}J_{x^{-}}^{\alpha }f(a)\right] .  \notag
\end{eqnarray}%
Similarly, for $x\neq b,$ we get%
\begin{eqnarray}
&&\dint\limits_{0}^{1}t\left( \left( \alpha +1\right) \lambda -t^{\alpha
}\right) f^{\prime \prime }\left( tx+\left( 1-t\right) b\right) dt
\label{2-1b} \\
&=&-\left( \left( \alpha +1\right) \lambda -1\right) \frac{f^{\prime }\left(
x\right) }{b-x}+\frac{\alpha +1}{\left( b-x\right) ^{2}}\left[ \left(
1-\lambda \right) f(x)+\lambda f(b)-\frac{\Gamma \left( \alpha +1\right) }{%
\left( b-x\right) ^{\alpha }}J_{x^{+}}^{\alpha }f(b)\right]  \notag
\end{eqnarray}%
$.$Multiplying both sides of (\ref{2-1a}) and (\ref{2-1b}) by $\frac{\left(
x-a\right) ^{\alpha +2}}{\left( \alpha +1\right) \left( b-a\right) }$ and $%
\frac{\left( b-x\right) ^{\alpha +2}}{\left( \alpha +1\right) \left(
b-a\right) }$, respectively, and adding the resulting identities, we obtain
the desired result.

For $x=a$ and $x=b,$the identities%
\begin{equation*}
I_{f}\left( a,\lambda ,\alpha ,a,b\right) =\frac{\left( b-a\right) ^{\alpha
+1}}{\left( \alpha +1\right) }\dint\limits_{0}^{1}t\left( \left( \alpha
+1\right) \lambda -t^{\alpha }\right) f^{\prime \prime }\left( ta+\left(
1-t\right) b\right) dt,
\end{equation*}%
and%
\begin{equation*}
I_{f}\left( b,\lambda ,\alpha ,a,b\right) =\frac{\left( b-a\right) ^{\alpha
+1}}{\left( \alpha +1\right) }\dint\limits_{0}^{1}t\left( \left( \alpha
+1\right) \lambda -t^{\alpha }\right) f^{\prime \prime }\left( tb+\left(
1-t\right) a\right) dt,
\end{equation*}%
can be proved by performing an integration by parts twice in the integrals
from the right side and changing the variable.
\end{proof}

\begin{theorem}
\label{2.1.1}Let $f:$ $I\subset \lbrack 0,\infty )\rightarrow 
\mathbb{R}
$ be twice differentiable function on $I^{\circ }$ such that $f^{\prime
\prime }\in L[a,b]$, where $a,b\in I^{\circ }$ with $a<b$. If $|f^{\prime
\prime }|^{q}$ is $s-$convex on $[a,b]$ for some fixed $s\in \left( 0,1%
\right] $ and$\ q\geq 1$, then for $x\in \lbrack a,b]$, $\lambda \in \left[
0,1\right] $ and $\alpha >0$  the following inequality for fractional
integrals holds%
\begin{eqnarray}
&&\left\vert I_{f}\left( x,\lambda ,\alpha ,a,b\right) \right\vert   \notag
\\
&\leq &C_{1}^{1-\frac{1}{q}}\left( \alpha ,\lambda \right) \left\{ \frac{%
\left( x-a\right) ^{\alpha +2}}{\left( \alpha +1\right) \left( b-a\right) }%
\left( \left\vert f^{\prime \prime }\left( x\right) \right\vert
^{q}C_{2}\left( \alpha ,\lambda ,s\right) +\left\vert f^{\prime \prime
}\left( a\right) \right\vert ^{q}C_{3}\left( \alpha ,\lambda ,s\right)
\right) ^{\frac{1}{q}}\right.   \label{2-2} \\
&&+\left. \frac{\left( b-x\right) ^{\alpha +2}}{\left( \alpha +1\right)
\left( b-a\right) }\left( \left\vert f^{\prime \prime }\left( x\right)
\right\vert ^{q}C_{2}\left( \alpha ,\lambda ,s\right) +\left\vert f^{\prime
\prime }\left( b\right) \right\vert ^{q}C_{3}\left( \alpha ,\lambda
,s\right) \right) ^{\frac{1}{q}}\right\} ,  \notag
\end{eqnarray}%
where 
\begin{eqnarray*}
C_{1}\left( \alpha ,\lambda \right)  &=&\left\{ 
\begin{array}{cc}
\frac{\alpha \left[ \left( \alpha +1\right) \lambda \right] ^{1+\frac{2}{%
\alpha }}}{\alpha +2}-\frac{\left( \alpha +1\right) \lambda }{2}+\frac{1}{%
\alpha +2}, & 0\leq \lambda \leq \frac{1}{\alpha +1} \\ 
\frac{\left( \alpha +1\right) \lambda }{2}-\frac{1}{\alpha +2}, & \frac{1}{%
\alpha +1}<\lambda \leq 1%
\end{array}%
\right. , \\
C_{2}\left( \alpha ,\lambda ,s\right)  &=&\left\{ 
\begin{array}{cc}
\frac{2\alpha \left[ \left( \alpha +1\right) \lambda \right] ^{\frac{\alpha
+s+2}{\alpha }}}{\left( s+2\right) \left( \alpha +s+2\right) }-\frac{\left(
\alpha +1\right) \lambda }{s+2}+\frac{1}{\alpha +s+2}, & 0\leq \lambda \leq 
\frac{1}{\alpha +1} \\ 
\frac{\left( \alpha +1\right) \lambda }{s+2}-\frac{1}{\alpha +s+2}, & \frac{1%
}{\alpha +1}<\lambda \leq 1%
\end{array}%
\right. , \\
C_{3}\left( \alpha ,\lambda ,s\right)  &=&\left\{ 
\begin{array}{cc}
\begin{array}{c}
\beta \left( \alpha +2,s+1\right) -\left( \alpha +1\right) \lambda \beta
\left( 2,s+1\right)  \\ 
2\left( \alpha +1\right) \lambda \beta \left( \left[ \left( \alpha +1\right)
\lambda \right] ^{\frac{1}{\alpha }};2,s+1\right) -2\beta \left( \left[
\left( \alpha +1\right) \lambda \right] ^{\frac{1}{\alpha }};\alpha
+2,s+1\right) 
\end{array}%
, & 0\leq \lambda \leq \frac{1}{\alpha +1} \\ 
\left( \alpha +1\right) \lambda \beta \left( 2,s+1\right) -\beta \left(
\alpha +2,s+1\right) , & \frac{1}{\alpha +1}<\lambda \leq 1%
\end{array}%
\right. .
\end{eqnarray*}
\end{theorem}

\begin{proof}
From Lemma \ref{2.1}, property of the modulus and using the power-mean
inequality we have%
\begin{eqnarray}
&&\left\vert I_{f}\left( x,\lambda ,\alpha ,a,b\right) \right\vert \leq 
\frac{\left( x-a\right) ^{\alpha +2}}{\left( \alpha +1\right) \left(
b-a\right) }\dint\limits_{0}^{1}t\left\vert \left( \alpha +1\right) \lambda
-t^{\alpha }\right\vert \left\vert f^{\prime \prime }\left( tx+\left(
1-t\right) a\right) \right\vert dt  \notag \\
&&+\frac{\left( b-x\right) ^{\alpha +2}}{\left( \alpha +1\right) \left(
b-a\right) }\dint\limits_{0}^{1}t\left\vert \left( \alpha +1\right) \lambda
-t^{\alpha }\right\vert \left\vert f^{\prime \prime }\left( tx+\left(
1-t\right) b\right) \right\vert dt  \notag \\
&\leq &\frac{\left( x-a\right) ^{\alpha +2}}{\left( \alpha +1\right) \left(
b-a\right) }\left( \dint\limits_{0}^{1}t\left\vert \left( \alpha +1\right)
\lambda -t^{\alpha }\right\vert dt\right) ^{1-\frac{1}{q}}  \notag \\
&&\times \left( \dint\limits_{0}^{1}t\left\vert \left( \alpha +1\right)
\lambda -t^{\alpha }\right\vert \left\vert f^{\prime \prime }\left(
tx+\left( 1-t\right) a\right) \right\vert ^{q}dt\right) ^{\frac{1}{q}} 
\notag \\
&&+\frac{\left( b-x\right) ^{\alpha +2}}{\left( \alpha +1\right) \left(
b-a\right) }\left( \dint\limits_{0}^{1}t\left\vert t\left( \alpha +1\right)
\lambda -t^{\alpha }\right\vert dt\right) ^{1-\frac{1}{q}}  \notag \\
&&\times \left( \dint\limits_{0}^{1}t\left\vert \left( \alpha +1\right)
\lambda -t^{\alpha }\right\vert \left\vert f^{\prime \prime }\left(
tx+\left( 1-t\right) b\right) \right\vert ^{q}dt\right) ^{\frac{1}{q}}.
\label{2-2a}
\end{eqnarray}%
Since $\left\vert f^{\prime }\right\vert ^{q}$ is $s-$convex on $[a,b]$ we
get%
\begin{eqnarray}
\dint\limits_{0}^{1}t\left\vert \left( \alpha +1\right) \lambda -t^{\alpha
}\right\vert \left\vert f^{\prime \prime }\left( tx+\left( 1-t\right)
a\right) \right\vert ^{q}dt &\leq &\dint\limits_{0}^{1}t\left\vert \left(
\alpha +1\right) \lambda -t^{\alpha }\right\vert \left( t^{s}\left\vert
f^{\prime \prime }\left( x\right) \right\vert ^{q}+\left( 1-t\right)
^{s}\left\vert f^{\prime \prime }\left( a\right) \right\vert ^{q}\right) dt 
\notag \\
&=&\left\vert f^{\prime \prime }\left( x\right) \right\vert ^{q}C_{2}\left(
\alpha ,\lambda ,s\right) +\left\vert f^{\prime \prime }\left( a\right)
\right\vert ^{q}C_{3}\left( \alpha ,\lambda ,s\right) ,  \label{2-2b}
\end{eqnarray}%
\begin{eqnarray}
\dint\limits_{0}^{1}t\left\vert \left( \alpha +1\right) \lambda -t^{\alpha
}\right\vert \left\vert f^{\prime \prime }\left( tx+\left( 1-t\right)
b\right) \right\vert ^{q}dt &\leq &\dint\limits_{0}^{1}t\left\vert \left(
\alpha +1\right) \lambda -t^{\alpha }\right\vert \left( t^{s}\left\vert
f^{\prime \prime }\left( x\right) \right\vert ^{q}+\left( 1-t\right)
^{s}\left\vert f^{\prime \prime }\left( b\right) \right\vert ^{q}\right) dt 
\notag \\
&=&\left\vert f^{\prime \prime }\left( x\right) \right\vert ^{q}C_{2}\left(
\alpha ,\lambda ,s\right) +\left\vert f^{\prime \prime }\left( b\right)
\right\vert ^{q}C_{3}\left( \alpha ,\lambda ,s\right) ,  \label{2-2c}
\end{eqnarray}%
where we use the fact that%
\begin{eqnarray*}
&&C_{3}\left( \alpha ,\lambda ,s\right) \\
&=&\dint\limits_{0}^{1}t\left\vert \left( \alpha +1\right) \lambda
-t^{\alpha }\right\vert \left( 1-t\right) ^{s}dt \\
&=&\left\{ 
\begin{array}{cc}
\begin{array}{c}
\left( \alpha +1\right) \lambda \dint\limits_{0}^{\left[ \left( \alpha
+1\right) \lambda \right] ^{\frac{1}{\alpha }}}t\left( 1-t\right)
^{s}dt-\dint\limits_{0}^{\left[ \left( \alpha +1\right) \lambda \right] ^{%
\frac{1}{\alpha }}}t^{\alpha +1}\left( 1-t\right) ^{s}dt \\ 
-\left( \alpha +1\right) \lambda \dint\limits_{\left[ \left( \alpha
+1\right) \lambda \right] ^{\frac{1}{\alpha }}}^{1}t\left( 1-t\right)
^{s}dt+\dint\limits_{\left[ \left( \alpha +1\right) \lambda \right] ^{\frac{1%
}{\alpha }}}^{1}t^{\alpha +1}\left( 1-t\right) ^{s}dt%
\end{array}%
, & 0\leq \lambda \leq \frac{1}{\alpha +1} \\ 
\left( \alpha +1\right) \lambda \dint\limits_{0}^{1}t\left( 1-t\right)
^{s}dt-\dint\limits_{0}^{1}t^{\alpha +1}\left( 1-t\right) ^{s}dt, & \frac{1}{%
\alpha +1}<\lambda \leq 1%
\end{array}%
\right. \\
&=&\left\{ 
\begin{array}{cc}
\begin{array}{c}
\beta \left( \alpha +2,s+1\right) -\left( \alpha +1\right) \lambda \beta
\left( 2,s+1\right) \\ 
2\left( \alpha +1\right) \lambda \beta \left( \left[ \left( \alpha +1\right)
\lambda \right] ^{\frac{1}{\alpha }};2,s+1\right) -2\beta \left( \left[
\left( \alpha +1\right) \lambda \right] ^{\frac{1}{\alpha }};\alpha
+2,s+1\right)%
\end{array}%
, & 0\leq \lambda \leq \frac{1}{\alpha +1} \\ 
\left( \alpha +1\right) \lambda \beta \left( 2,s+1\right) -\beta \left(
\alpha +2,s+1\right) , & \frac{1}{\alpha +1}<\lambda \leq 1%
\end{array}%
\right. ,
\end{eqnarray*}%
\begin{eqnarray*}
&&C_{2}\left( \alpha ,\lambda ,s\right) \\
&=&\dint\limits_{0}^{1}t\left\vert \left( \alpha +1\right) \lambda
-t^{\alpha }\right\vert t^{s}dt \\
&=&\left\{ 
\begin{array}{cc}
\frac{2\alpha \left[ \left( \alpha +1\right) \lambda \right] ^{\frac{\alpha
+s+2}{\alpha }}}{\left( s+2\right) \left( \alpha +s+2\right) }-\frac{\left(
\alpha +1\right) \lambda }{s+2}+\frac{1}{\alpha +s+2}, & 0\leq \lambda \leq 
\frac{1}{\alpha +1} \\ 
\frac{\left( \alpha +1\right) \lambda }{s+2}-\frac{1}{\alpha +s+2}, & \frac{1%
}{\alpha +1}<\lambda \leq 1%
\end{array}%
\right.
\end{eqnarray*}%
and 
\begin{eqnarray}
&&\dint\limits_{0}^{1}t\left\vert \left( \alpha +1\right) \lambda -t^{\alpha
}\right\vert dt  \notag \\
&=&\left\{ 
\begin{array}{cc}
\frac{\alpha \left[ \left( \alpha +1\right) \lambda \right] ^{1+\frac{2}{%
\alpha }}}{\alpha +2}-\frac{\left( \alpha +1\right) \lambda }{2}+\frac{1}{%
\alpha +2}, & 0\leq \lambda \leq \frac{1}{\alpha +1} \\ 
\frac{\left( \alpha +1\right) \lambda }{2}-\frac{1}{\alpha +2}, & \frac{1}{%
\alpha +1}<\lambda \leq 1%
\end{array}%
\right.  \label{2-2d}
\end{eqnarray}%
Hence, If we use (\ref{2-2b}), (\ref{2-2c}) and (\ref{2-2d}) in (\ref{2-2a}%
), we obtain the desired result. This completes the proof.
\end{proof}

\begin{corollary}
\label{2.a}In Theorem \ref{2.1.1},

(a) If we choose $s=1,$ then we get:%
\begin{eqnarray*}
&&\left\vert I_{f}\left( x,\lambda ,\alpha ,a,b\right) \right\vert \\
&\leq &C_{1}^{1-\frac{1}{q}}\left( \alpha ,\lambda \right) \left\{ \frac{%
\left( x-a\right) ^{\alpha +1}}{b-a}\left( \left\vert f^{\prime \prime
}\left( x\right) \right\vert ^{q}C_{2}\left( \alpha ,\lambda ,1\right)
+\left\vert f^{\prime \prime }\left( a\right) \right\vert ^{q}C_{3}\left(
\alpha ,\lambda ,1\right) \right) ^{\frac{1}{q}}\right. \\
&&+\left. \frac{\left( b-x\right) ^{\alpha +1}}{b-a}\left( \left\vert
f^{\prime \prime }\left( x\right) \right\vert ^{q}C_{2}\left( \alpha
,\lambda ,1\right) +\left\vert f^{\prime \prime }\left( b\right) \right\vert
^{q}C_{3}\left( \alpha ,\lambda ,1\right) \right) ^{\frac{1}{q}}\right\} .
\end{eqnarray*}

(b) If we choose $q=1,$ then we get:%
\begin{eqnarray*}
\left\vert I_{f}\left( x,\lambda ,\alpha ,a,b\right) \right\vert &\leq
&\left\{ \frac{\left( x-a\right) ^{\alpha +1}}{b-a}\left( \left\vert
f^{\prime \prime }\left( x\right) \right\vert C_{2}\left( \alpha ,\lambda
,s\right) +\left\vert f^{\prime \prime }\left( a\right) \right\vert
C_{3}\left( \alpha ,\lambda ,s\right) \right) \right. \\
&&+\left. \frac{\left( b-x\right) ^{\alpha +1}}{b-a}\left( \left\vert
f^{\prime \prime }\left( x\right) \right\vert C_{2}\left( \alpha ,\lambda
,s\right) +\left\vert f^{\prime \prime }\left( b\right) \right\vert
C_{3}\left( \alpha ,\lambda ,s\right) \right) \right\} .
\end{eqnarray*}

(c) If we choose $x=\frac{a+b}{2},$ then we get:%
\begin{eqnarray}
&&\left\vert \frac{2^{\alpha -1}}{\left( b-a\right) ^{\alpha -1}}I_{f}\left( 
\frac{a+b}{2},\lambda ,\alpha ,a,b\right) \right\vert  \notag \\
&=&\left\vert \left( 1-\lambda \right) f\left( \frac{a+b}{2}\right) +\lambda
\left( \frac{f(a)+f(b)}{2}\right) -\frac{\Gamma \left( \alpha +1\right)
2^{\alpha -1}}{\left( b-a\right) ^{\alpha }}\left[ J_{\left( \frac{a+b}{2}%
\right) ^{-}}^{\alpha }f(a)+J_{\left( \frac{a+b}{2}\right) ^{+}}^{\alpha
}f(b)\right] \right\vert  \notag \\
&\leq &\frac{\left( b-a\right) ^{2}}{8\left( \alpha +1\right) }C_{1}^{1-%
\frac{1}{q}}\left( \alpha ,\lambda \right) \left\{ \left( \left\vert
f^{\prime \prime }\left( \frac{a+b}{2}\right) \right\vert ^{q}C_{2}\left(
\alpha ,\lambda ,s\right) +\left\vert f^{\prime \prime }\left( a\right)
\right\vert ^{q}C_{3}\left( \alpha ,\lambda ,s\right) \right) ^{\frac{1}{q}%
}\right.  \notag \\
&&+\left. \left( \left\vert f^{\prime \prime }\left( \frac{a+b}{2}\right)
\right\vert ^{q}C_{2}\left( \alpha ,\lambda ,s\right) +\left\vert f^{\prime
\prime }\left( b\right) \right\vert ^{q}C_{3}\left( \alpha ,\lambda
,s\right) \right) ^{\frac{1}{q}}\right\} .  \label{2-2e}
\end{eqnarray}

(d) If we choose $x=\frac{a+b}{2}$ and $\lambda =\frac{1}{3},$ then we get:%
\begin{eqnarray*}
&&\left\vert \frac{2^{\alpha -1}}{\left( b-a\right) ^{\alpha -1}}I_{f}\left( 
\frac{a+b}{2},\frac{1}{3},\alpha ,a,b\right) \right\vert \\
&=&\left\vert \frac{1}{6}\left[ f(a)+4f\left( \frac{a+b}{2}\right) +f(b)%
\right] -\frac{\Gamma \left( \alpha +1\right) 2^{\alpha -1}}{\left(
b-a\right) ^{\alpha }}\left[ J_{\left( \frac{a+b}{2}\right) ^{-}}^{\alpha
}f(a)+J_{\left( \frac{a+b}{2}\right) ^{+}}^{\alpha }f(b)\right] \right\vert
\\
&\leq &\frac{\left( b-a\right) ^{2}}{8\left( \alpha +1\right) }C_{1}^{1-%
\frac{1}{q}}\left( \alpha ,\frac{1}{3}\right) \left\{ \left( \left\vert
f^{\prime \prime }\left( \frac{a+b}{2}\right) \right\vert ^{q}C_{2}\left(
\alpha ,\frac{1}{3},s\right) +\left\vert f^{\prime \prime }\left( a\right)
\right\vert ^{q}C_{3}\left( \alpha ,\frac{1}{3},s\right) \right) ^{\frac{1}{q%
}}\right. \\
&&\left. +\left( \left\vert f^{\prime \prime }\left( \frac{a+b}{2}\right)
\right\vert ^{q}C_{2}\left( \alpha ,\frac{1}{3},s\right) +\left\vert
f^{\prime \prime }\left( b\right) \right\vert ^{q}C_{3}\left( \alpha ,\frac{1%
}{3},s\right) \right) ^{\frac{1}{q}}\right\} .
\end{eqnarray*}

(e) If we choose $x=\frac{a+b}{2},\ \lambda =\frac{1}{3},$ and $\alpha =1$,
then we get:%
\begin{eqnarray*}
&&\left\vert I_{f}\left( \frac{a+b}{2},\frac{1}{3},1,a,b\right) \right\vert
\\
&=&\left\vert \frac{1}{6}\left[ f(a)+4f\left( \frac{a+b}{2}\right) +f(b)%
\right] -\frac{1}{\left( b-a\right) }\dint\limits_{a}^{b}f(x)dx\right\vert \\
&\leq &\frac{\left( b-a\right) ^{2}}{162}\left( \frac{81}{8}\right) ^{\frac{1%
}{q}}\left\{ \left( \left\vert f^{\prime \prime }\left( \frac{a+b}{2}\right)
\right\vert ^{q}C_{2}\left( 1,\frac{1}{3},s\right) +\left\vert f^{\prime
\prime }\left( a\right) \right\vert ^{q}C_{3}\left( 1,\frac{1}{3},s\right)
\right) ^{\frac{1}{q}}\right. \\
&&\left. +\left( \left\vert f^{\prime \prime }\left( \frac{a+b}{2}\right)
\right\vert ^{q}C_{2}\left( 1,\frac{1}{3},s\right) +\left\vert f^{\prime
\prime }\left( b\right) \right\vert ^{q}C_{3}\left( 1,\frac{1}{3},s\right)
\right) ^{\frac{1}{q}}\right\} .
\end{eqnarray*}

(f) If we choose $x=\frac{a+b}{2}$ and$\ \lambda =0,$ then we get:%
\begin{eqnarray*}
&&\left\vert \frac{2^{\alpha -1}}{\left( b-a\right) ^{\alpha -1}}I_{f}\left( 
\frac{a+b}{2},0,\alpha ,a,b\right) \right\vert \\
&=&\left\vert f\left( \frac{a+b}{2}\right) -\frac{\Gamma \left( \alpha
+1\right) 2^{\alpha -1}}{\left( b-a\right) ^{\alpha }}\left[ J_{\left( \frac{%
a+b}{2}\right) ^{-}}^{\alpha }f(a)+J_{\left( \frac{a+b}{2}\right)
^{+}}^{\alpha }f(b)\right] \right\vert \\
&\leq &\frac{\left( b-a\right) ^{2}}{8\left( \alpha +1\right) }\left( \frac{1%
}{\alpha +2}\right) ^{1-\frac{1}{q}}\left\{ \left[ \frac{\left\vert
f^{\prime \prime }\left( \frac{a+b}{2}\right) \right\vert ^{q}}{\alpha +s+2}%
+\left\vert f^{\prime \prime }\left( a\right) \right\vert ^{q}\beta \left(
\alpha +2,s+1\right) \right] ^{\frac{1}{q}}\right. \\
&&\left. +\left[ \frac{\left\vert f^{\prime \prime }\left( \frac{a+b}{2}%
\right) \right\vert ^{q}}{\alpha +s+2}+\left\vert f^{\prime \prime }\left(
b\right) \right\vert ^{q}\beta \left( \alpha +2,s+1\right) \right] ^{\frac{1%
}{q}}\right\} .
\end{eqnarray*}

(g) If we choose $x=\frac{a+b}{2},\ \lambda =0,$ and $\alpha =1$, then we
get: 
\begin{eqnarray*}
&&\left\vert I_{f}\left( \frac{a+b}{2},0,1,a,b\right) \right\vert \\
&=&\left\vert f\left( \frac{a+b}{2}\right) -\frac{1}{\left( b-a\right) }%
\dint\limits_{a}^{b}f(x)dx\right\vert \\
&\leq &\frac{\left( b-a\right) ^{2}}{16}\left( \frac{1}{3}\right) ^{1-\frac{1%
}{q}}\left\{ \left[ \frac{\left\vert f^{\prime \prime }\left( \frac{a+b}{2}%
\right) \right\vert ^{q}}{s+3}+\left\vert f^{\prime \prime }\left( a\right)
\right\vert ^{q}\beta \left( 3,s+1\right) \right] ^{\frac{1}{q}}\right. \\
&&\left. +\left[ \frac{\left\vert f^{\prime \prime }\left( \frac{a+b}{2}%
\right) \right\vert ^{q}}{s+3}+\left\vert f^{\prime \prime }\left( b\right)
\right\vert ^{q}\beta \left( 3,s+1\right) \right] ^{\frac{1}{q}}\right\} .
\end{eqnarray*}

(h) If we choose$\ x=\frac{a+b}{2}$ and $\lambda =1,$ then we get:%
\begin{eqnarray*}
&&\left\vert \frac{2^{\alpha -1}}{\left( b-a\right) ^{\alpha -1}}I_{f}\left( 
\frac{a+b}{2},1,\alpha ,a,b\right) \right\vert \\
&=&\left\vert \frac{f(a)+f(b)}{2}-\frac{\Gamma \left( \alpha +1\right)
2^{\alpha -1}}{\left( b-a\right) ^{\alpha }}\left[ J_{\left( \frac{a+b}{2}%
\right) ^{-}}^{\alpha }f(a)+J_{\left( \frac{a+b}{2}\right) ^{+}}^{\alpha
}f(b)\right] \right\vert \\
&\leq &\frac{\left( b-a\right) ^{2}}{8\left( \alpha +1\right) }\left( \frac{%
\alpha \left( \alpha +3\right) }{2\left( \alpha +2\right) }\right) ^{1-\frac{%
1}{q}} \\
&&\times \left\{ \left[ \frac{\alpha \left( \alpha +s+3\right) \left\vert
f^{\prime \prime }\left( \frac{a+b}{2}\right) \right\vert ^{q}}{\left(
s+2\right) \left( \alpha +s+2\right) }+\left\vert f^{\prime \prime }\left(
a\right) \right\vert ^{q}\left( \left( \alpha +1\right) \beta \left(
2,s+1\right) -\beta \left( \alpha +2,s+1\right) \right) \right] ^{\frac{1}{q}%
}\right. \\
&&\left. +\left[ \frac{\alpha \left( \alpha +s+3\right) \left\vert f^{\prime
\prime }\left( \frac{a+b}{2}\right) \right\vert ^{q}}{\left( s+2\right)
\left( \alpha +s+2\right) }+\left\vert f^{\prime \prime }\left( b\right)
\right\vert ^{q}\left( \left( \alpha +1\right) \beta \left( 2,s+1\right)
-\beta \left( \alpha +2,s+1\right) \right) \right] ^{\frac{1}{q}}\right\} .
\end{eqnarray*}

(i) If we choose$\ x=\frac{a+b}{2},\lambda =1$ and $\alpha =1$, then we get:%
\begin{eqnarray*}
&&\left\vert I_{f}\left( \frac{a+b}{2},1,1,a,b\right) \right\vert \\
&=&\left\vert \frac{f(a)+f(b)}{2}-\frac{1}{\left( b-a\right) }%
\dint\limits_{a}^{b}f(x)dx\right\vert \\
&\leq &\frac{\left( b-a\right) ^{2}}{16}\left( \frac{2}{3}\right) ^{1-\frac{1%
}{q}} \\
&&\times \left\{ \left[ \frac{\left( s+4\right) \left\vert f^{\prime \prime
}\left( \frac{a+b}{2}\right) \right\vert ^{q}}{\left( s+2\right) \left(
s+3\right) }+\left\vert f^{\prime \prime }\left( a\right) \right\vert
^{q}\left( 2\beta \left( 2,s+1\right) -\beta \left( 3,s+1\right) \right) %
\right] ^{\frac{1}{q}}\right. \\
&&\left. +\left[ \frac{\left( s+4\right) \left\vert f^{\prime \prime }\left( 
\frac{a+b}{2}\right) \right\vert ^{q}}{\left( s+2\right) \left( s+3\right) }%
+\left\vert f^{\prime \prime }\left( b\right) \right\vert ^{q}\left( 2\beta
\left( 2,s+1\right) -\beta \left( 3,s+1\right) \right) \right] ^{\frac{1}{q}%
}\right\} .
\end{eqnarray*}
\end{corollary}

\begin{remark}
In \ (c) of Corollary \ref{2.a}, if we choose $\alpha =1,$ then the
inequality (\ref{2-2e}) reduces to the inequality (\ref{1-3}).
\end{remark}

\begin{remark}
In (e) of Corollary \ref{2.a}, if we choose $s=1,$ we have the following
Simpson type inequality%
\begin{eqnarray*}
&&\left\vert \frac{1}{6}\left[ f(a)+4f\left( \frac{a+b}{2}\right) +f(b)%
\right] -\frac{1}{\left( b-a\right) }\dint\limits_{a}^{b}f(x)dx\right\vert \\
&\leq &\frac{\left( b-a\right) ^{2}}{16}\left( \frac{8}{81}\right) ^{1-\frac{%
1}{q}}\left\{ \left( \frac{59}{972}\left\vert f^{\prime \prime }\left( \frac{%
a+b}{2}\right) \right\vert ^{q}+\frac{27}{972}\left\vert f^{\prime \prime
}\left( a\right) \right\vert ^{q}\right) ^{\frac{1}{q}}\right. \\
&&\left. +\left( \frac{59}{972}\left\vert f^{\prime \prime }\left( \frac{a+b%
}{2}\right) \right\vert ^{q}+\frac{27}{972}\left\vert f^{\prime \prime
}\left( b\right) \right\vert ^{q}\right) ^{\frac{1}{q}}\right\} .
\end{eqnarray*}%
which is better than the inequality (\ref{1-2}).
\end{remark}

\begin{theorem}
\label{2.2}Let $f:$ $I\subset \lbrack 0,\infty )\rightarrow 
\mathbb{R}
$ be twice differentiable function on $I^{\circ }$ such that $f^{\prime
\prime }\in L[a,b]$, where $a,b\in I^{\circ }$ with $a<b$. If $|f^{\prime
\prime }|^{q}$ is $s-$convex on $[a,b]$ for some fixed $s\in \left( 0,1%
\right] $ and $q>1$, then for $x\in \lbrack a,b]$, $\lambda \in \left[ 0,1%
\right] $ and $\alpha >0$  the following inequality for fractional integrals
holds%
\begin{eqnarray}
&&\left\vert I_{f}\left( x,\lambda ,\alpha ,a,b\right) \right\vert 
\label{2-3} \\
&\leq &C_{4}^{\frac{1}{p}}\left( \alpha ,\lambda ,p\right) \left\{ \frac{%
\left( x-a\right) ^{\alpha +2}}{\left( \alpha +1\right) \left( b-a\right) }%
\left( \frac{\left\vert f^{\prime \prime }\left( x\right) \right\vert
^{q}+\left\vert f^{\prime \prime }\left( a\right) \right\vert ^{q}}{s+1}%
\right) ^{\frac{1}{q}}\right.   \notag \\
&&+\left. \frac{\left( b-x\right) ^{\alpha +2}}{\left( \alpha +1\right)
\left( b-a\right) }\left( \frac{\left\vert f^{\prime \prime }\left( x\right)
\right\vert ^{q}+\left\vert f^{\prime \prime }\left( b\right) \right\vert
^{q}}{s+1}\right) ^{\frac{1}{q}}\right\} ,  \notag
\end{eqnarray}%
where $p=\frac{q}{q-1}$ and 
\begin{eqnarray*}
&&C_{4}\left( \alpha ,\lambda ,p\right)  \\
&=&\left\{ 
\begin{array}{cc}
\frac{1}{p\left( \alpha +1\right) +1}, & \lambda =0 \\ 
\left[ 
\begin{array}{c}
\frac{\left[ \left( \alpha +1\right) \lambda \right] ^{\frac{1+\left( \alpha
+1\right) p}{\alpha }}}{\alpha }\beta \left( \frac{1+p}{\alpha },1+p\right) 
\\ 
+\frac{\left[ 1-\left( \alpha +1\right) \lambda \right] ^{p+1}}{\alpha
\left( p+1\right) }._{2}F_{1}\left( 1-\frac{1+p}{\alpha },1;p+2;1-\left(
\alpha +1\right) \lambda \right) 
\end{array}%
\right] , & 0<\lambda \leq \frac{1}{\alpha +1} \\ 
\frac{\left[ \left( \alpha +1\right) \lambda \right] ^{\frac{p\left( \alpha
+1\right) +1}{\alpha }}}{\alpha }\beta \left( \frac{1}{\left( \alpha
+1\right) \lambda };\frac{1+p}{\alpha },1+p\right) , & \frac{1}{\alpha +1}%
<\lambda \leq 1%
\end{array}%
\right. .
\end{eqnarray*}
\end{theorem}

\begin{proof}
From Lemma \ref{2.1}, property of the modulus and using the H\"{o}lder
inequality we have%
\begin{eqnarray}
&&\left\vert I_{f}\left( x,\lambda ,\alpha ,a,b\right) \right\vert \leq 
\frac{\left( x-a\right) ^{\alpha +2}}{\left( \alpha +1\right) \left(
b-a\right) }\dint\limits_{0}^{1}t\left\vert \left( \alpha +1\right) \lambda
-t^{\alpha }\right\vert \left\vert f^{\prime \prime }\left( tx+\left(
1-t\right) a\right) \right\vert dt  \notag \\
&&+\frac{\left( b-x\right) ^{\alpha +2}}{\left( \alpha +1\right) \left(
b-a\right) }\dint\limits_{0}^{1}t\left\vert \left( \alpha +1\right) \lambda
-t^{\alpha }\right\vert \left\vert f^{\prime \prime }\left( tx+\left(
1-t\right) b\right) \right\vert dt  \notag \\
&\leq &\frac{\left( x-a\right) ^{\alpha +2}}{\left( \alpha +1\right) \left(
b-a\right) }\left( \dint\limits_{0}^{1}t^{p}\left\vert \left( \alpha
+1\right) \lambda -t^{\alpha }\right\vert ^{p}dt\right) ^{\frac{1}{p}}\times
\left( \dint\limits_{0}^{1}\left\vert f^{\prime \prime }\left( tx+\left(
1-t\right) a\right) \right\vert ^{q}dt\right) ^{\frac{1}{q}}  \notag \\
&&+\frac{\left( b-x\right) ^{\alpha +2}}{\left( \alpha +1\right) \left(
b-a\right) }\left( \dint\limits_{0}^{1}t^{p}\left\vert \left( \alpha
+1\right) \lambda -t^{\alpha }\right\vert ^{p}dt\right) ^{\frac{1}{p}}\times
\left( \dint\limits_{0}^{1}\left\vert f^{\prime \prime }\left( tx+\left(
1-t\right) b\right) \right\vert ^{q}dt\right) ^{\frac{1}{q}}  \label{2-3a}
\end{eqnarray}%
Since $\left\vert f^{\prime }\right\vert ^{q}$ is $s-$convex on $[a,b]$ we
get%
\begin{eqnarray}
\dint\limits_{0}^{1}\left\vert f^{\prime \prime }\left( tx+\left( 1-t\right)
a\right) \right\vert ^{q}dt &\leq &\dint\limits_{0}^{1}\left(
t^{s}\left\vert f^{\prime \prime }\left( x\right) \right\vert ^{q}+\left(
1-t\right) ^{s}\left\vert f^{\prime \prime }\left( a\right) \right\vert
^{q}\right) dt  \notag \\
&=&\frac{\left\vert f^{\prime \prime }\left( x\right) \right\vert
^{q}+\left\vert f^{\prime \prime }\left( a\right) \right\vert ^{q}}{s+1},
\label{2-3b}
\end{eqnarray}%
\begin{eqnarray}
\dint\limits_{0}^{1}\left\vert f^{\prime \prime }\left( tx+\left( 1-t\right)
b\right) \right\vert ^{q}dt &\leq &\dint\limits_{0}^{1}\left(
t^{s}\left\vert f^{\prime \prime }\left( x\right) \right\vert ^{q}+\left(
1-t\right) ^{s}\left\vert f^{\prime \prime }\left( b\right) \right\vert
^{q}\right) dt  \notag \\
&=&\frac{\left\vert f^{\prime \prime }\left( x\right) \right\vert
^{q}+\left\vert f^{\prime \prime }\left( b\right) \right\vert ^{q}}{s+1},
\label{2-3c}
\end{eqnarray}%
and%
\begin{eqnarray}
&&\dint\limits_{0}^{1}t^{p}\left\vert \left( \alpha +1\right) \lambda
-t^{\alpha }\right\vert ^{p}dt  \label{2-3d} \\
&=&\left\{ 
\begin{array}{cc}
\dint\limits_{0}^{1}t^{\left( \alpha +1\right) p}dt & \lambda =0 \\ 
\dint\limits_{0}^{\left[ \left( \alpha +1\right) \lambda \right] ^{\frac{1}{%
\alpha }}}t^{p}\left[ \left( \alpha +1\right) \lambda -t^{\alpha }\right]
^{p}dt+\dint\limits_{\left[ \left( \alpha +1\right) \lambda \right] ^{\frac{1%
}{\alpha }}}^{1}t^{p}\left[ t^{\alpha }-\left( \alpha +1\right) \lambda %
\right] ^{p}dt, & 0<\lambda \leq \frac{1}{\alpha +1} \\ 
\dint\limits_{0}^{1}t^{p}\left[ \left( \alpha +1\right) \lambda -t^{\alpha }%
\right] ^{p}dt, & \frac{1}{\alpha +1}<\lambda \leq 1%
\end{array}%
\right.  \notag \\
&=&\left\{ 
\begin{array}{cc}
\frac{1}{p\left( \alpha +1\right) +1}, & \lambda =0 \\ 
\left[ 
\begin{array}{c}
\frac{\left[ \left( \alpha +1\right) \lambda \right] ^{\frac{1+\left( \alpha
+1\right) p}{\alpha }}}{\alpha }\beta \left( \frac{1+p}{\alpha },1+p\right)
\\ 
+\frac{\left[ 1-\left( \alpha +1\right) \lambda \right] ^{p+1}}{\alpha
\left( p+1\right) }._{2}F_{1}\left( 1-\frac{1+p}{\alpha },1;p+2;1-\left(
\alpha +1\right) \lambda \right)%
\end{array}%
\right] , & 0<\lambda \leq \frac{1}{\alpha +1} \\ 
\frac{\left[ \left( \alpha +1\right) \lambda \right] ^{\frac{1+\left( \alpha
+1\right) p}{\alpha }}}{\alpha }\beta \left( \frac{1}{\left( \alpha
+1\right) \lambda };\frac{1+p}{\alpha },1+p\right) , & \frac{1}{\alpha +1}%
<\lambda \leq 1%
\end{array}%
\right.  \notag
\end{eqnarray}%
Hence, If we use (\ref{2-3b}), (\ref{2-3c}) and (\ref{2-3d}) in (\ref{2-3a}%
), we obtain the desired result. This completes the proof.
\end{proof}

\begin{corollary}
\label{2.2b}In Theorem \ref{2.2}

(a) If we choose $s=1,$ ten we get:%
\begin{eqnarray*}
&&\left\vert I_{f}\left( x,\lambda ,\alpha ,a,b\right) \right\vert \\
&\leq &C_{4}^{\frac{1}{p}}\left( \alpha ,\lambda ,p\right) \left\{ \frac{%
\left( x-a\right) ^{\alpha +2}}{\left( \alpha +1\right) \left( b-a\right) }%
\left( \frac{\left\vert f^{\prime \prime }\left( x\right) \right\vert
^{q}+\left\vert f^{\prime \prime }\left( a\right) \right\vert ^{q}}{2}%
\right) ^{\frac{1}{q}}\right. \\
&&+\left. \frac{\left( b-x\right) ^{\alpha +2}}{\left( \alpha +1\right)
\left( b-a\right) }\left( \frac{\left\vert f^{\prime \prime }\left( x\right)
\right\vert ^{q}+\left\vert f^{\prime \prime }\left( b\right) \right\vert
^{q}}{2}\right) ^{\frac{1}{q}}\right\} .
\end{eqnarray*}

(b) If we choose $x=\frac{a+b}{2},$ then we get:%
\begin{eqnarray*}
&&\left\vert I_{f}\left( x,\lambda ,\alpha ,a,b\right) \right\vert \\
&\leq &C_{4}^{\frac{1}{p}}\left( \alpha ,\lambda ,p\right) \frac{\left(
b-a\right) ^{\alpha +1}}{\left( \alpha +1\right) 2^{\alpha +2}}\left\{
\left( \frac{\left\vert f^{\prime \prime }\left( \frac{a+b}{2}\right)
\right\vert ^{q}+\left\vert f^{\prime \prime }\left( a\right) \right\vert
^{q}}{s+1}\right) ^{\frac{1}{q}}\right. \\
&&+\left. \left( \frac{\left\vert f^{\prime \prime }\left( \frac{a+b}{2}%
\right) \right\vert ^{q}+\left\vert f^{\prime \prime }\left( b\right)
\right\vert ^{q}}{s+1}\right) ^{\frac{1}{q}}\right\} .
\end{eqnarray*}

(c) If we choose $x=\frac{a+b}{2},\ \lambda =\frac{1}{3},$then we get:%
\begin{eqnarray*}
&&\left\vert \frac{1}{6}\left[ f(a)+4f\left( \frac{a+b}{2}\right) +f(b)%
\right] -\frac{\Gamma \left( \alpha +1\right) 2^{\alpha -1}}{\left(
b-a\right) ^{\alpha }}\left[ J_{\left( \frac{a+b}{2}\right) ^{-}}^{\alpha
}f(a)+J_{\left( \frac{a+b}{2}\right) ^{+}}^{\alpha }f(b)\right] \right\vert
\\
&\leq &C_{4}^{\frac{1}{p}}\left( \alpha ,\frac{1}{3},p\right) \frac{\left(
b-a\right) ^{2}}{8\left( \alpha +1\right) }\left\{ \left( \frac{\left\vert
f^{\prime \prime }\left( \frac{a+b}{2}\right) \right\vert ^{q}+\left\vert
f^{\prime \prime }\left( a\right) \right\vert ^{q}}{s+1}\right) ^{\frac{1}{q}%
}\right. \\
&&+\left. \left( \frac{\left\vert f^{\prime \prime }\left( \frac{a+b}{2}%
\right) \right\vert ^{q}+\left\vert f^{\prime \prime }\left( b\right)
\right\vert ^{q}}{s+1}\right) ^{\frac{1}{q}}\right\} .
\end{eqnarray*}

(d) If we choose $x=\frac{a+b}{2},\ \lambda =\frac{1}{3},$ and $\alpha =1,$%
then we get:%
\begin{eqnarray*}
&&\left\vert \frac{1}{6}\left[ f(a)+4f\left( \frac{a+b}{2}\right) +f(b)%
\right] -\frac{1}{\left( b-a\right) }\dint\limits_{a}^{b}f(x)dx\right\vert \\
&\leq &\frac{\left( b-a\right) ^{2}}{16}C_{4}^{\frac{1}{p}}\left( 1,\frac{1}{%
3},p\right) \left\{ \left( \frac{\left\vert f^{\prime \prime }\left( \frac{%
a+b}{2}\right) \right\vert ^{q}+\left\vert f^{\prime \prime }\left( a\right)
\right\vert ^{q}}{s+1}\right) ^{\frac{1}{q}}\right. \\
&&+\left. \left( \frac{\left\vert f^{\prime \prime }\left( \frac{a+b}{2}%
\right) \right\vert ^{q}+\left\vert f^{\prime \prime }\left( b\right)
\right\vert ^{q}}{s+1}\right) ^{\frac{1}{q}}\right\} ,
\end{eqnarray*}%
where 
\begin{equation*}
C_{4}\left( 1,\frac{1}{3},p\right) =\left( \frac{2}{3}\right) ^{1+2p}\beta
\left( 1+p,1+p\right) +\left( \frac{1}{3}\right) ^{1+p}._{2}F_{1}\left(
-p,1;p+2;\frac{1}{3}\right) .
\end{equation*}

(e) If we choose $x=\frac{a+b}{2}$ and$\ \lambda =0,$then we get:%
\begin{eqnarray*}
&&\left\vert f\left( \frac{a+b}{2}\right) -\frac{\Gamma \left( \alpha
+1\right) 2^{\alpha -1}}{\left( b-a\right) ^{\alpha }}\left[ J_{\left( \frac{%
a+b}{2}\right) ^{-}}^{\alpha }f(a)+J_{\left( \frac{a+b}{2}\right)
^{+}}^{\alpha }f(b)\right] \right\vert \\
&\leq &\frac{\left( b-a\right) ^{2}}{16}\left( \frac{1}{p\left( \alpha
+1\right) +1}\right) ^{\frac{1}{p}}\left\{ \left( \frac{\left\vert f^{\prime
\prime }\left( \frac{a+b}{2}\right) \right\vert ^{q}+\left\vert f^{\prime
\prime }\left( a\right) \right\vert ^{q}}{s+1}\right) ^{\frac{1}{q}}\right.
\\
&&+\left. \left( \frac{\left\vert f^{\prime \prime }\left( \frac{a+b}{2}%
\right) \right\vert ^{q}+\left\vert f^{\prime \prime }\left( b\right)
\right\vert ^{q}}{s+1}\right) ^{\frac{1}{q}}\right\} .
\end{eqnarray*}

(f) If we choose $x=\frac{a+b}{2}$ and $\lambda =1,$then we get:%
\begin{eqnarray*}
&&\left\vert \frac{f(a)+f(b)}{2}-\frac{\Gamma \left( \alpha +1\right)
2^{\alpha -1}}{\left( b-a\right) ^{\alpha }}\left[ J_{\left( \frac{a+b}{2}%
\right) ^{-}}^{\alpha }f(a)+J_{\left( \frac{a+b}{2}\right) ^{+}}^{\alpha
}f(b)\right] \right\vert \\
&\leq &\frac{\left( b-a\right) ^{2}}{16}C_{4}^{\frac{1}{p}}\left( \alpha
,1,p\right) \left\{ \left( \frac{\left\vert f^{\prime \prime }\left( \frac{%
a+b}{2}\right) \right\vert ^{q}+\left\vert f^{\prime \prime }\left( a\right)
\right\vert ^{q}}{s+1}\right) ^{\frac{1}{q}}\right. \\
&&+\left. \left( \frac{\left\vert f^{\prime \prime }\left( \frac{a+b}{2}%
\right) \right\vert ^{q}+\left\vert f^{\prime \prime }\left( b\right)
\right\vert ^{q}}{s+1}\right) ^{\frac{1}{q}}\right\} ,
\end{eqnarray*}%
where%
\begin{equation*}
C_{4}\left( \alpha ,1,p\right) =\frac{\left( 1+\alpha \right) ^{\frac{%
p\left( \alpha +1\right) +1}{\alpha }}}{\alpha }\beta \left( \frac{1}{%
1+\alpha };\frac{1+p}{\alpha },1+p\right) .
\end{equation*}

(g) If we choose $x=\frac{a+b}{2},\lambda =1$ and $\alpha =1,$then we get:$\ 
$%
\begin{eqnarray*}
&&\left\vert \frac{f(a)+f(b)}{2}-\frac{1}{\left( b-a\right) }%
\dint\limits_{a}^{b}f(x)dx\right\vert \\
&\leq &\frac{\left( b-a\right) ^{2}}{4}\left( 2\beta \left( \frac{1}{2}%
;1+p,1+p\right) \right) ^{\frac{1}{p}}\left\{ \left( \frac{\left\vert
f^{\prime \prime }\left( \frac{a+b}{2}\right) \right\vert ^{q}+\left\vert
f^{\prime \prime }\left( a\right) \right\vert ^{q}}{s+1}\right) ^{\frac{1}{q}%
}\right. \\
&&+\left. \left( \frac{\left\vert f^{\prime \prime }\left( \frac{a+b}{2}%
\right) \right\vert ^{q}+\left\vert f^{\prime \prime }\left( b\right)
\right\vert ^{q}}{s+1}\right) ^{\frac{1}{q}}\right\} .
\end{eqnarray*}
\end{corollary}

\begin{remark}
In \ (b) of Corollary \ref{2.2b}, if we take $\lambda =\frac{1}{2}-\frac{1}{r%
},r\geq 2,$ and $\alpha =1,$ then the inequality (\ref{2-3}) reduces to the
inequality (\ref{1-4}).
\end{remark}

\end{document}